\documentclass[USenglish]{article}
\usepackage[utf8]{inputenc}
\usepackage[big,online]{dgruyter}
\usepackage{lmodern}
\usepackage{fix-cm}
\usepackage{mathrsfs,amsmath}
\usepackage{microtype}
\usepackage[numbers,square,sort&compress]{natbib}
\usepackage{hyperref}
\hypersetup{colorlinks=false}
\usepackage{xcolor}
\usepackage{mathtools}
\usepackage{anyfontsize}
\usepackage{mathrsfs}

\theoremstyle{dgthm}
\newtheorem{theorem}{Theorem}[section]
\newtheorem{lemma}{Lemma}[section]
\newtheorem{remark}{Remark}[section]

\newtheorem{corollary}{Corollary}[section]
\DeclarePairedDelimiter{\norm}{\lVert}{\rVert}
\newcommand{\lr}[1]{\left(#1\right)}
\AtBeginDocument{%
}
\begin{document}

\title{Double weighted sum involving $\mathrm{GL}(2)$ Fourier coefficients}
\author[1]{Himanshi Chanana}
\author*[2]{Mohd Harun}
\runningauthor{Himanshi Chanana and Mohd Harun}
\affil[1]{\protect\raggedright Department of Mathematics and Statistics, Indian Institute of Technology Kanpur, 
Kalyanpur, Kanpur Nagar-208016, India, email: hchanana20@gmail.com, ORCID-0000 0001 5643 8712}
\affil[2]{\protect \raggedright Department of Science and Mathematics, Indian Institute of Information Technology Guwahati-781015,
India, email: harunmalikjmi@gmail.com, ORCID-0000 0003 0052 4122}

\abstract{
 This article proves non-trivial estimates for a bilinear sum involving the Fourier coefficients of a Hecke-holomorphic or Hecke-Maass cusp form for $\mathrm{SL}(2,\mathbb{Z})$. As corollaries, we draw interesting results related to non-trivial bounds of different shifted convolution sums and summatory functions.} 
\maketitle

\section{Introduction}
Let $\alpha(n), \beta(n)$ and $\gamma(n)$ be the coefficients of certain $L$-functions or more general coefficients of arithmetic interest. Let $X$ and $Y$ be real numbers such that $X, Y \geqslant 1$ and $Y \leqslant X$. A fundamental problem in number theory is to understand the correlation sums of the form:

\begin{equation} \label{SCSs}
  \mathsf{S_1}:=  \sum_{1 \leqslant n \leqslant X} \alpha(n) \gamma(n+m) \ \text{ and } \ \mathsf{S_2}:= \mathop{\sum_{1 \leqslant n \leqslant X}\sum_{ 1 \leqslant m \leqslant Y}} \alpha(n)\beta(m) \gamma(n+m).
\end{equation}
\noindent The study of these types of sums play a pivotal role in the understanding of some of the prominent problems in analytic number theory, such as the moments of $L$-functions, subconvexity, the Gauss circle problem, and the Quantum Unique Ergodicity (QUE) conjecture (see, e.g., \cite{Blomer1}, \cite{DFI}, \cite{HM}, \cite{HB}, \cite{HS}, \cite{LLY}, and the references therein). Shifted convolution sums, i.e., sums of the form $\mathsf{S}_1$, are directly related to many unsolved problems in number theory, such as the twin prime conjecture (when $\alpha = \gamma = \Lambda$, the von Mangoldt function) and the Chowla conjecture for two-point correlations (when $\alpha = \gamma = \mu$, the M\"{o}bius function). Additionally, many authors in the literature focus on the sum $\mathsf{S_1}$ with \(\alpha = \tau_{\ell}\) and \(\gamma = \tau_k\), where \(\tau_{\ell}\) and \(\tau_k\) are the \(\ell\)-th and \(k\)-th divisor functions, respectively, for \(\ell\) and $ k \in \mathbb{N}$. This is related to the classical additive divisor problem, which seeks an asymptotic formula for \(\mathsf{S_1}\) when \(\ell\) and \( k \geqslant 2\). It is conjectured \cite{Ng SCS conjecture} that
\vspace{0.1cm}
\begin{equation*}
    \sum_{n \leqslant X} \tau_{\ell}(n) \tau_k(n+m) \sim C_{\ell,k}(m) X(\log X)^{\ell+k-2}, \ \ \ \ \ \ \ \ \ \text{as} \ X \rightarrow \infty.
\end{equation*}
\vspace{0.1cm}

\noindent The analogy is deeper because $\tau_k(n)$ appears as the $n$-th Fourier coefficient of certain Eisenstein series for $\mathrm{SL}(k,\mathbb{Z})$. For the case of $\ell=k=2$, the strongest results are due to Meurman \cite{meurman}. 
\vspace{0.1cm}

\noindent One of the challenging problems is understanding the average behavior of any arithmetic function over sparse sequences.  If $\alpha, \beta$ and $\gamma$ are the same multiplicative functions then the sums $\mathsf{S_1}$ and $\mathsf{S_2}$ can also be interpreted as studying the average behavior of an arithmetic function over a sparse sequence of the form $p(x)=x(x + m)$ and $p(x,y)=xy(x+y)$, respectively. In this article, we study cancellations between such sums, focusing on the Fourier coefficients of $\mathrm{GL}(2)$ forms by using a variant of the delta method due to Duke et al \cite[Eq. 20.157]{R1}. In this direction, some of the classical results are due to Ingham \cite{ingham}, Scourfield \cite{scour}, and Hooley \cite{hooley}. They established an asymptotic expansion for the average behavior of $\tau(n)$ over sequences of the form $p(x) = x^2+a$ for $a \in \mathbb{Z}$. Erd$\ddot{o}$s \cite{erdos} established the following result for any irreducible polynomial $p(x) \in \mathbb{Z}[x]$:

\begin{equation*}
    X \log{X} \ll \sum_{1 \leqslant n \leqslant X} \tau(p(n)) \ll X \log{X}.
\end{equation*}


  \vspace{0.2cm}   

\noindent To prove a non-trivial bound for such sums, we study general bilinear sums involving the Fourier coefficients of $\mathrm{GL}(2)$ forms and two arbitrary complex sequences. Friedlander and Iwaniec  \cite{FI} studied the $\mathrm{GL}(1)$ analogue of a similar problem in $1993$. For a prime $p$, and $\chi$, a non-principal character modulo $p$, they established a non-trivial bound for the following sum:

\begin{equation} \label{gl1 result}
    \sum_{n \in \mathcal{N}} \sum_{m \in \mathcal{M}} \mathsf{a}(n) \mathsf{b}(m) \chi(n+m),
\end{equation}
\noindent where $\mathsf{a}(n)$ and $ \mathsf{b}(n)$ are bounded arithmetic functions, and $\mathcal{N}$, $\mathcal{M}$ are finite sets of integers. They proved several estimates for this sum, which, in particular, provide a modified proof of the Pólya-Vinogradov and Burgess inequalities. A decade earlier, while studying the $\mathrm{GL}(2)$ analogue of the famous Titchmarsh's divisor problem, Pitt \cite{pitt} derived a non-trivial bound for a similar bilinear sum. He considered the following sum
\begin{equation*}
  \mathfrak{B}_{X,Y} =  \sum_{1 \leqslant n \leqslant X}\sum_{1 \leqslant m \leqslant Y} \mathsf{a}(n)\mathsf{b}(m)\mathcal{A}_f(nm-1),
\end{equation*}
where $\mathcal{A}_f(n)$ is the normalized $n$-th Fourier coefficient of $f \in S_k(\mathrm{SL}(2,\mathbb{Z}))$, the space of holomorphic cusp forms, at the cusp $\infty$ (see \autoref{GL2 forms} for details). He proved

\begin{equation*}
   \mathfrak{B}_{X,Y} \ll_{f,\varepsilon} \norm{\mathsf{a}}_2 \norm{\mathsf{b}}_2 \big\{ X^{1/2} Y^{3/8}  + X^{3/8} Y^{3/4} \big\} (XY)^{(k-1)/2 + \varepsilon}. 
\end{equation*}
Here, $\norm{.}_2$ denotes the usual $\ell^2$-norm.
\medskip

\noindent In the first result of this article, we consider an analogous sum as Friedlander and Iwaniec for general $\mathrm{GL}(2)$ Fourier coefficients, which extends the result in \autoref{gl1 result} to $\mathrm{GL}(2)$ Fourier coefficients. Let $f$ be a Hecke-holomorphic or Hecke-Maass cusp form for the full modular group $\mathrm{SL}(2,\mathbb{Z})$ with normalized Fourier coefficients $\mathcal{A}_f(n)$. We define:

\begin{equation} \label{define S(X,Y)}
    \mathscr{S}_{X,Y} := \sum_{1 \leqslant n \leqslant X}\sum_{1 \leqslant m \leqslant Y} \mathsf{a}(n)\mathsf{b}(m)\mathcal{A}_f(n+m),
\end{equation}
where $\mathsf{a}(n)$ and $\mathsf{b}(n)$ are any complex sequences. It follows from the Rankin-Selberg theory that the Fourier coefficients $\mathcal{A}_f (n)  $\textquotesingle s are bounded on average (see \autoref{bound for gl2 coeff}). 
Now, applying the Cauchy-Schwarz inequality first in the $n$-sum and then in the $m$-sum, and using \autoref{bound for gl2 coeff}, we obtain that, for any $\varepsilon>0$, the sum $\mathscr{S}_{X,Y}$ in \autoref{define S(X,Y)} is trivially bounded by $\norm{\mathsf{a}}_2 \norm{\mathsf{b}}_2  (XY)^{1/2+\varepsilon}$. As our primary result, we prove the following theorem.

\begin{theorem}\label{th1}
Let $\mathsf{a}(n)$ and $ \mathsf{b}(n)$ be any two sequences of complex numbers, and $\mathcal{A}_f(n)$ be the normalized $n$-th Fourier coefficient of a Hecke-holomorphic or Hecke-Maass cusp form $f$ for the full modular group $\mathrm{SL}(2,\mathbb{Z})$. Let $Y < X^{1-\varepsilon}$ for any arbitrarily small $\varepsilon >0$, we have
\begin{align*}
\mathscr{S}_{X,Y} \ll_{f, \varepsilon } \,\norm{\mathsf{a}}_2 \, \norm{\mathsf{b}}_2\,\frac{ X^{3/8}}{ Y^{3/4}}\, (X Y)^{1/2 + \varepsilon },   
\end{align*}
where $\lVert . \rVert_2$ denotes the Euclidean norm. This bound for $ \mathscr{S}_{X,Y}$ is non-trivial within the range $ X^{1/2} < Y <  X^{1-\varepsilon}$. 
\end{theorem}
\begin{corollary}
    Let $ X$ and $ Y$ be defined as in Theorem \ref{th1}. Set $\mathsf{a}(n)= \Lambda(n),$ $\mu(n)$ or $\tau_k(n)$ and $\mathsf{b}(n) =1$ for all $n \in \mathbb{N}$. Then we have
    \begin{equation*}
        \sum_{1 \leqslant n \leqslant X}\,\sum_{1 \leqslant m \leqslant Y} \, \mathsf{a}(n) \mathcal{A}_f(n+m) \ll_{f, \varepsilon } \frac{ X^{3/8}}{ Y^{3/4}}\, (X Y)^{1 + \varepsilon }.
    \end{equation*}
This bound is non-trivial in the range $ X^{1/2} < Y <  X^{1-\varepsilon}$.    
\end{corollary}
\begin{remark}
   The above corollary provides a non-trivial bound for the $\mathrm{GL}(2)$ analog of Titchmarsh's divisor problem on average over the shift. Furthermore, if we take \(\mathsf{a}(n)\) to be \(\mathrm{GL}(3)\) Fourier coefficients, then the first power saving result for \(\mathrm{GL}(3) \times \mathrm{GL}(2)\) shifted convolution sum was obtained by Munshi \cite{RM gl3xgl2}. Subsequently, Sun \cite{Sun avg over shift} established a non-trivial bound for the smooth sum by considering additional averaging over the shift. The second author \cite{HSS} of the present article also established a non-trivial bound of the same strength as of Sun \cite{Sun avg over shift} for the weighted average version of $\mathrm{GL}(3) \times \mathrm{GL}(2)$ shifted convolution sum.
\end{remark}


 The next corollary establishes an on-average non-trivial bound for the shifted convolution sum, earlier studied by L\"{u} et al \cite{Guangshi et al} for holomorphic cusp forms $f$ of weight $k$ and level $N$. Let $\ell \in \mathbb{N}$, and $r_\ell(n) = |\{(n_1,n_2,...,n_\ell) \in \mathbb{Z}^\ell : n_1^2+n_2^2+...n_\ell^2 =n\}|$ denote the $n$-th Fourier coefficient of the modular form which is $\ell$-th power of the classical Jacobi theta series, $\theta(z)= \sum_{n \in \mathbb{Z}} e(n^2z)$. It is well-known that $r_\ell(n) \asymp n^{\ell/2-1}$ holds for $\ell>4$.  They proved
\begin{equation} \label{guangshi bound rln}
    \sum_{n \leqslant X} \mathcal{A}_f(n+m) r_{\ell}(n) \ll_{f,\ell,\varepsilon} X^{\ell/2 - \vartheta_{\ell} +\varepsilon},
\end{equation}
where for $N=1$, $\vartheta_2 = \frac{1}{6}, \vartheta_3 =\frac{1}{4}, \vartheta_4 = \vartheta_5 = \frac{1}{2}$ and $\vartheta_\ell = \frac{2}{3}$, for $\ell \geqslant 6$. Later, Sun in \cite{Sun theta} generalizes the result in \autoref{guangshi bound rln} to more general $\mathrm{GL}(2)$ cusp form that is a holomorphic form of weight $k$, level $N$, and nebentypus $\chi_N$, or Maass forms of weight $0$ or $1$, level $N$, nebentypus $\chi_N$ with  Laplace eigenvalue $1/4+\nu^2$.

\begin{corollary} \label{coro rln}
    Let $ X$ and $ Y$ be defined as in Theorem \ref{th1}. Set $\mathsf{a}(n)= r_{\ell}(n)$ and $\mathsf{b}(n) = 1$ with $\ell > 4$. We have
    \begin{equation*}
        \sum_{1 \leqslant n \leqslant X} \sum_{1 \leqslant m \leqslant Y} r_{\ell}(n) \mathcal{A}_f(n+m) \ll_{f,\ell, \varepsilon} \frac{ X^{3/8}}{ Y^{3/4}}\, \left(X^{\ell/2} Y\right)^{1 + \varepsilon }.
    \end{equation*}
    This bound is non-trivial in the range $X^{1/2}< Y< X^{1-\varepsilon}$.
\end{corollary}
Furthermore, Zhang in \cite{zhang} has established an asymptotic expansion over the diagonal quadratic form in the case of the divisor function. More generally, we obtain the following non-trivial bound. 

\begin{corollary}
    Let $ X$ and $ Y$ be defined as in Theorem \ref{th1}. Set $\mathsf{a}(n)= r_{\ell_1}(n)$ and $\mathsf{b}(n) = r_{\ell_2}(n)$ with $\ell_1, \ell_2 > 4$. We have
    \begin{equation*}
         \sum_{1 \leqslant n \leqslant X} \sum_{1 \leqslant m \leqslant Y} r_{\ell_1}(n) r_{\ell_2}(m) \mathcal{A}_f(n+m) \ll_{f,\ell_1,\ell_2,\varepsilon} \frac{ X^{3/8}}{ Y^{3/4}}\, \left(X^{\ell_1/2} Y^{\ell_2/2}\right)^{1 + \varepsilon }.
    \end{equation*}
    This bound is non-trivial in the range $X^{1/2}< Y< X^{1-\varepsilon}$.
\end{corollary}

The next result establishes cancellation in the summatory function involving $\mathcal{A}_f(n)$ over polynomials of the form $x+y^k$ for $k\geqslant 2$. Earlier, Yu \cite{yu} considered the sum over $x^2+y^2$ in the case of the divisor function, and Acharya \cite{acharya} proved the following bound

\begin{equation*}
    \mathop{\sum\sum}_{\substack{n,m\in \mathbb{Z} \\ n^2+m^2 \leqslant X}} \mathcal{A}_f(n^2+m^2) \ll_{f,\varepsilon} X^{1/2+\varepsilon}, 
\end{equation*}
for a holomorphic cusp form $f$  of even weight $k$ for the group $\Gamma_0(4N)$ with $N\in \mathbb{N}$. As a second result of this article, we prove cancellations in a general sum and get some interesting consequences in particular cases. We prove the following theorem. 

\begin{theorem}\label{th2}
Let $\mathsf{c}(n)$ be any complex sequence and $\mathcal{A}_f(n)$ be defined as in Theorem \ref{th1}.  Let $X$ be a variable, and $\mathsf{S} \subset[X,2X]$ such that $|\mathsf{S}| \asymp X^{1/k}$ with a natural number $k > 1$. We have
\begin{align*}
  \sum_{X \leqslant n \leqslant 2X}\,\sum_{m \in \mathsf{S}} \, \mathsf{c}(n)\, \mathcal{A}_{f}(n+m) \ll_{f, \varepsilon } \norm{\mathsf{c}}_{2} X^{1/2} \left({\norm{\gamma_{\ell}}_2 } \right)^{1/2^{\ell +1}}X^{{\frac{3}{ 2^{\ell +4}} - \frac{1}{4k} + \varepsilon  }},
\end{align*}
where $\ell$ is least such that $2^{\ell+1} > k (1-\delta_1)$ for some $\delta_1 \geqslant  1/2$ and 
\begin{align*}
 |\gamma_{\ell}(m)| =  \mathop{\sum_{m_1 \in \mathsf{S}}\,\sum_{m_2 \in \mathsf{S}}\,...\,\sum_{m_{2^{\ell+1}} \in \mathsf{S}}}_{m \,=\,(m_1-m_2) + (m_3 - m_4) \,+\, ...\, + \,(m_{2^{\ell+1}-1}- m_{2^{\ell+1}}) } \,1. 
 \end{align*}
\end{theorem}
\begin{corollary} \label{cor thm2}
    We have
    \begin{equation*}
      \sum_{X \leqslant n \leqslant 2X} \sum_{X^{1/k} \leqslant m \leqslant (2X)^{1/k}}  \mathcal{A}_f(n+m^k) \ll_{f,\varepsilon} X^{1+ \frac{1}{k} - \left( \frac{1}{4k} + \frac{1}{2^{\ell+4}}\right) + \varepsilon},
    \end{equation*}
    where $\ell \in \mathbb{N}$ is least such that $2^{\ell+1} > \text{max}\left\{k(1-\delta_1),\, k \log k\right\} $ for some $\delta_1 \geqslant 1/2$.
\end{corollary}
\begin{proof}
   Taking $\mathsf{c}(n) = 1$ for all $n \in \mathbb{N}$ and $\mathsf{S} =\{ n \in \mathbb{N} : X \leqslant n^k \leqslant 2X \}$ with $k\geqslant 2$. We get our desired result using Lemma \ref{gamma} and Theorem \ref{th2}.  
\end{proof}
\begin{remark}
    For $\mathsf{c}(n) = \Lambda(n)$, $\mu(n)$ or $\tau_k(n)$ in Theorem \ref{th2}, we can obtain the same strength bound as in Corollary \ref{cor thm2}. 
\end{remark}

\section{Preliminaries} \label{preliminaries}

\subsection{DFI Delta Method}\label{delta}

This paper will use a version of the delta method due to Duke, Friedlander, and Iwaniec. We will use the expansion $(20.157)$ given in Chapter $20$ of \cite{R1}. Let $\delta : \mathbb{Z}\to \{0,1\}$ be defined by
\[
\delta(n,m)=
\begin{cases}
	1 &\text{if}\,\,n=m \\
	0 &\text{otherwise}
\end{cases}
\]
Then for $n,m\in\mathbb{Z}\cap [-2L,2L]$, we have
\begin{equation}\label{a1}
	\delta(n,m)=\frac{1}{Q} \sum_{1 \leqslant q \leqslant Q} \frac{1}{q} \, \ \sideset{}{^\star} \sum_{a\bmod q}e\left(\frac{(n-m) a}{q}\right)\int_{\mathbb{R}} \psi(q,x)e\left(\frac{(n-m) x}{q	Q}\right) d x, 
\end{equation} where $Q=2L^{1/2}$. Following properties of the function $\psi(q,u)$ are of our interest (see $(20.158)$ and $(20.159)$ of \cite{R1}, and \cite[Lemma 15]{R2})

\begin{align}\label{Delta}
	&\psi(q,x)=1+h(q,x),\,\,\,\,\text{with}\,\,\,\,h(q,x)=\mathcal{O}\left(\frac{Q}{q }\left(\frac{q}{Q}+|x|\right)^A\right),\\
	&\psi(q,x)\ll |x|^{-A}, \label{2nd property}\\
	&x^j \frac{\partial^j}{ \partial x^j} \psi(q, x) \ll \  \min \left\lbrace \frac{Q}{q}, \frac{1}{|x|} \right\rbrace \  \log Q,
\end{align}
for any $A>1$,  $j \geqslant 1$.  In particular, the second property (\autoref{2nd property}) implies that the effective range of integral in \autoref{a1} is $[-L^{\epsilon}, L^{\epsilon}]$. It also follows that if $q \ll Q^{1- \epsilon}$ and $  x \ll Q^{- \epsilon} $, then  $ \psi(q, x)$ can
be replaced by $1$ at the cost of a negligible error term. If $ q \gg Q^{1- \epsilon}$, then we get $ x^j \frac{\partial^j}{ \partial x^j} \psi(q, x) \ll Q^{\epsilon}$, for any $ j \geqslant 1$. If $ q \ll  Q^{1- \epsilon}$  and $  Q^{- \epsilon} \ll |x| \ll   Q^{ \epsilon}$, then $ x^j \frac{\partial^j}{ \partial x^j} \psi(q, x) \ll Q^{\epsilon}$, for any $ j \geqslant 1$. Hence, we can view $  \psi(q, x)$ as a nice weight function in all cases.\\ 

\subsection{\texorpdfstring{$\mathrm{SL}(2,\mathbb{Z})$ Automorphic forms}{}} \label{GL2 forms}
 In this section, we shall briefly recall some fundamental facts about the $\mathrm{SL}(2,\mathbb{Z})$ automorphic forms, Voronoi summation formula, and some other results used in our analysis. Let $f$ be a primitive holomorphic Hecke eigenform of integral weight $k$ on the full modular group $\mathrm{SL}(2,\mathbb{Z})$. The normalized (i.e., $\mathcal{A}_f(1)=1$) Fourier expansion of $f$ at the cusp $\infty$ is given by
\begin{equation*}
    f(z) = \sum_{n=1}^\infty \mathcal{A}_f(n)n^{(k-1)/2}e(nz). 
\end{equation*}
Analogously, let $f$ be a primitive Maass cusp form on the group $\mathrm{SL}(2,\mathbb{Z})$ with Laplacian eigenvalue $\frac{1}{4} + \nu^2$. The normalised Fourier expansion of $f$ at the cusp $\infty$ is given by
\begin{equation*}
    f(z) = 2\sqrt{y} \sum_{n \neq 0} \mathcal{A}_f(n) K_{i\nu}(2\pi|n|y) e(nx),
\end{equation*}
where $K_{i \nu}$ denotes the $K$-Bessel function and $\mathcal{A}_f(1)=1$. By the Rankin-Selberg theory, the Fourier coefficients $\mathcal{A}_f(n)$ satisfy the following bound on average:
\begin{equation} \label{bound for gl2 coeff}
    \sum_{n \leqslant X} |\mathcal{A}_f(n)|^2 = c_fX + \mathcal{O}(X^{3/5}),
\end{equation}
for some constant $c_f>0$. The Ramanujan-Petersson conjecture predicts that $\mathcal{A}_f(n) \ll n^\varepsilon$. This has been proved by Deligne in the case of holomorphic cusp forms, where he proves that $\mathcal{A}_f(n) \leq d(n)$. For the case of Maass cusp form the best known result is due to Kim and Sarnak (see \cite{Kim}), which is $\mathcal{A}_f(n) \ll n^{7/64+\varepsilon}$. 
\begin{lemma} \label{voronoi}
	{\bf Voronoi summation formula for $\mathrm{GL}(2)$}:
Let $g$ be a compactly supported smooth function, supported on the interval $(0, \infty)$. Let $\mathcal{A}_f(n)$ be the $n$-th Fourier coefficients of a holomorphic or Hecke-Maass cusp form $f$ for $\mathrm{SL}(2, \mathbb{Z})$. Let $q \in \mathbb{N}$ and $a \in \mathbb{Z}$ such that $(a, q) = 1$ with $a \overline{a} \equiv 1 \bmod (q)$. Then we have
	\begin{equation}\label{varequation}
		\sum_{n=1}^\infty \mathcal{A}_f(n)\, e_q(an)\, g\left({n}\right) = \frac{Y}{q}\, \sum_{\pm} \sum_{n=1}^\infty\, \mathcal{A}_f(n) \,e_q(\mp\, \overline{a}n)\, G^{\pm} \left(\frac{n}{q^2}\right),
	\end{equation}
where $e_q(x)= e^{\frac{2 \pi i x}{q}}$. When $f$ is a holomorphic cusp form with weight $k$, we have
 \begin{align}\label{holo}
  G^{+} (y)= 2 \pi i^{k}\, \int_0^\infty g(x)  J_{k-1 }  \left( 4\pi \sqrt{xy}\right) dx ,\,\,\,\text{and} \,\,\,\,\, G^{-} (y) = 0. \end{align}
 When $f$ is a Maass cusp form with the Laplacian eigenvalue $1/4 + \nu^2$, $\nu >0$, we have
 \begin{align}\label{maass2}
		&G^{+} (y)= \frac{i\pi}{\sinh( \pi \nu)} \int_0^\infty g(x) \left\lbrace  J_{2i\nu } - J_{-2i\nu }\right\rbrace \left( 4\pi \sqrt{xy}\right) dx,  \,\,\,\,\text{and}
	\end{align} 
	\begin{align}\label{maass1}
 &G^{-} (y)= \epsilon_f\,4\cosh( \pi \nu) \int_0^\infty g(x)  K_{2i\nu }  \left( 4\pi \sqrt{xy}\right) dx,
 \end{align} 
where $\epsilon_f$ is the eigenvalue of $f$ under the reflection operator and $J_{2i\nu }$, $ K_{2i\nu }$ represent the Bessel functions of the first and second kind, respectively.
\end{lemma}
\begin{proof}
	See \cite[Appendix $A.4$]{r12}.
\end{proof}

\begin{remark}
 If the function $g$ is supported on $[Y, 2Y]$ satisfies $x^jg^{(j)}(x) \ll_j 1$, using the properties of the Bessel's functions given in Lemma \ref{L14} and then repeated integration by parts, we can easily deduce that the integral $G^{\pm}(n/q^2)$ is negligibly small if $n \gg q^2(qY)^{\epsilon}/Y$. Hence, the sum on the right-hand side of \autoref{varequation} is essentially supported on $n \ll q^2(qY)^{\epsilon}/Y$. For smaller values of $n$, we will use the trivial bound $G^{\pm}(n/q^2) \ll Y$.   
\end{remark}

\begin{lemma}\label{gamma}
For a natural number $k >1$, let $\mathsf{S} = \{m \in \mathbb{Z}: X \leqslant m^k \leqslant2X\} \subset [X, 2X]$. Let $\gamma_{\ell}(m)$ be an arithmetical function with
\begin{align*}
 |\gamma_{\ell}(m) | =  \mathop{\sum_{m_1 \in \mathsf{S}}\,\sum_{m_2 \in \mathsf{S}}\,...\,\sum_{m_{2^{\ell+1}} \in \mathsf{S}}}_{m \,=\,(m_1-m_2) + (m_3 - m_4) \,+\, ...\, + \,(m_{2^{\ell+1}-1}- m_{2^{\ell+1}}) } \,1\,. 
 \end{align*}
For $2^{\ell +1} > k \log k$, we have
\begin{align*}
\norm{\gamma_{\ell}}^2 =  \sum_{m \leqslant X}\,|\gamma_{\ell}(m)|^2 \,\ll\,  X^{2^{\ell+2}/k - 1} .  
\end{align*}
\end{lemma}
\begin{proof}
We can write
\begin{align*}
 \sum_{m \leqslant X}\,|\gamma_{\ell}(m)|^2 = \mathop{\mathop{\sum_{m \leqslant X}\,\sum_{m_1 \in \mathsf{S}}\,\sum_{m_2 \in \mathsf{S}}\,...\,\sum_{m_{2^{\ell+1}} \in \mathsf{S}}\,1 \,\sum_{n_1 \in \mathsf{S}}\,\sum_{n_2 \in \mathsf{S}}\,...\,\sum_{n_{2^{\ell+1}} \in \mathsf{S}} 
 1}_{m_1 - m_2 + m_3- m_4 + ... + m_{2^{\ell+1}-1} - m_{2^{\ell+1}} = m}}_{n_1 - n_2 + n_3 - n_4  ... + n_{2^{\ell+1}-1} - n_{2^{\ell+1}} = m}    
\end{align*}
\begin{align*}
    \ll\,&\,\,\, \mathop{\sum_{m \leqslant X}\,\sum_{m_1 \in \mathsf{S}}\,\sum_{m_2 \in \mathsf{S}}\,...\,\sum_{m_{2^{\ell+1}} \in \mathsf{S}}\, \,\sum_{n_1 \in \mathsf{S}}\,\sum_{n_2 \in \mathsf{S}}\,...\,\sum_{n_{2^{\ell+1}} \in \mathsf{S}}\, 
 1}_{m_1 + n_2 + m_3+ n_4 + ... + m_{2^{\ell+1}-1} + n_{2^{\ell+1}} = m = n_1 + m_2 + n_3 + m_4 + ... + n_{2^{\ell+1}-1} + m_{2^{\ell+1}}}
 =\,\,\, \mathop{\sum_{m \leqslant X}} |r_{\ell,k}(m)|^2,
\end{align*}
where $r_{\ell,k}(m)$ count the total number of representations of $m$ as the sum of $2^{\ell+1}$ many positive $k$-th powers. From \cite[Theorem $20.2$]{R1}, for $2^{\ell+1} > 2^k$, we have  
\begin{align*}
r_{\ell,k}(m) \asymp m^{2^{\ell+1}/k - 1} .   
\end{align*}
The condition $2^{\ell+1} > 2^k$ was further improved by Wooley to $2^{\ell+1} > k \log k$. Hence, we obtain
\begin{align*}
 \sum_{m \leqslant X}\,|\gamma_{\ell}(m)|^2 \ll\, X^{2^{\ell+2}/k - 1}.   \end{align*}
 This proves our lemma.
\end{proof}

\vspace{0.2cm}

\section{Proof of the Theorem \ref{th1}}
Let $X, Y \geqslant 1$ be two variables with $Y < X^{1-\epsilon}$. Let $\mathsf{a}(n)$, $\mathsf{b}(n)$ be any two complex sequences, and $\mathcal{A}_f(n)$ be the $n$-th Fourier coefficients of a holomorphic or Hecke Maass cusp form $f$ for $\mathrm{SL}(2, \mathbb{Z})$.  Our main object of study is the sum:
\begin{align}\label{S( X, Y)}
	\mathscr{S}_{ X, Y} = \sum_{1 \leqslant n \leqslant  X}\,\,\sum_{1 \leqslant m \leqslant Y} \mathsf{a}(n)\,\mathsf{b}(m)\, \mathcal{A}_f(n+m).
\end{align}
 On applying the Cauchy-Schwarz inequality to the sum over $n$, we arrive at:
\begin{align} \label{S( X, Y) with kappa}
		\mathscr{S}_{ X, Y} \ll  \,\norm{\mathsf{a}}_2 \, \left(\sum_{n =1}^\infty \, V_1\lr{\frac{n}{X}}\left|\sum_{1 \leqslant m \leqslant Y} \mathsf{b}(m)\,\mathcal{A}_f(n+m)\right|^2\right)^{1/2}  =  \norm{\mathsf{a}}_2 \, \left(\kappa(X,Y)\right)^{1/2} ,
\end{align}
	where $V_1$ is a positive smooth bump function supported on $[-2,2]$ with $V_1(x) = 1$ on $[-1,1]$, satisfying $x^j V_1^{(j)}(x) \ll 1$ for all $j \geq 1$, and $\kappa(X,Y)$ can be written as
	\begin{align*}
		\kappa(X,Y) = \sum_{1 \leqslant m_1 \leqslant Y} \,\sum_{1 \leqslant m_2 \leqslant Y}\, \mathsf{b}(m_1)\,\overline{\mathsf{b}(m_2)} \, \sum_{n =1}^\infty \, V_1\lr{\frac{n}{X}} \, \mathcal{A}_f(n+m_1)\,\overline{\mathcal{A}_f(n+m_2)}.     
	\end{align*}
	Changing the variable $n+m_1 \longrightarrow n$ and $m_2 - m_1 = m$, we get
	\begin{align*}
		\kappa(X,Y) = \sum_{1 \leqslant m_1 \leqslant Y} \;\sum_{1-m_1 \leqslant m \leqslant Y-m_1} \, \mathsf{b}(m_1)\,\overline{\mathsf{b}(m+m_1)} \, \sum_{n =1}^\infty \, V_1\lr{\frac{n-m_1}{X}} \, \mathcal{A}_f(n)\,\overline{\mathcal{A}_f(n+m)}. 
	\end{align*}
 By applying the Cauchy-Schwarz inequality to the  sum over $m$ and then $m_1$, we can write
	\begin{align} \label{kappa to omega}
		\kappa(X,Y) \ll \,\norm{\mathsf{b}}_2^2 \, \sqrt{ Y}  \left(\sum_{m\in \mathbb{Z}} W\lr{\frac{m}{Y}}\,\left|\ \sum_{n=1}^\infty V_1 \lr{\frac{n}{X_1}}\mathcal{A}_f(n) \overline{\mathcal{A}_f(n+m)}\ \right|^2 \right)^{1/2},
	\end{align}
where $W$ is a non-negative even smooth function supported on $[-2,2]$ with $W(x)=1$ on $[-1,1]$, satisfying $x^j W^{(j)}(x) \ll 1$ and $X_1 :=  X+Y \asymp  X$. Denote the inner sums as $\omega(X,Y)$, we have 
	\begin{align*}
		\omega(X,Y) = \,\sum_{n_1=1}^\infty \,\sum_{n_2 =1}^\infty\, V_1\lr{\frac{n_1}{X_1}} V_2\lr{\frac{n_2}{X_1}} \mathcal{A}_f(n_1)\,\overline{\mathcal{A}_f(n_2)}\, \sum_{m \in \mathbb{Z}} W\lr{\frac{m}{Y}}\, \overline{\mathcal{A}_f(n_1 +m)}\, \mathcal{A}_f(n_2 +m), 
	\end{align*}
where $V_2$ is another smooth bump function satisfying the same properties as $V_1$, is introduced to symmetrize the $n_1, n_2$ variables after Cauchy–Schwarz. Further, by changing the variable $ n_1 + m$ to $n_3$, we obtain,  
\begin{align*}
	\omega(X,Y) = &\,\sum_{n_1=1}^\infty \,\sum_{n_2 =1}^\infty\, V_1\lr{\frac{n_1}{X_1}} V_2\lr{\frac{n_2}{X_1}}\mathcal{A}_f(n_1)\,\overline{\mathcal{A}_f(n_2)}\, \\
    &\times \sum_{n_3 \in \mathbb{Z}}\, \overline{\mathcal{A}_f(n_3)}\,W\left(\frac{n_3-n_1}{ Y}\right)\, \mathcal{A}_f(n_2 +n_3-n_1).
\end{align*}
After the change of variables $n_3 = n_1 + m$, the support conditions on the
weights $V_1(n_1/X_1)$ and $W(m/Y)$ and $\mathcal{A}_f(n_3) = 0$ for $n_3 \leq 0$ imply that $1 \le n_3 \le 2X_1 + 2Y$.
We therefore introduce a smooth weight function
$V_3 \in C_c^\infty(\mathbb{R})$ such that $V_3(x)=1$ for $x \in [-1,1]$,
$\operatorname{supp}(V_3) \subset [-2,2]$, satisfying $x^j V_3^{(j)}(x) \ll 1$, and define
$
X_2 := 2X_1 + 2Y \asymp X$, we obtain
\begin{align*}
	\omega(X,Y) = &\,\sum_{n_1=1}^\infty \,\sum_{n_2 =1}^\infty\, V_1\lr{\frac{n_1}{X_1}}
    V_2\lr{\frac{n_2}{X_1}}\mathcal{A}_f(n_1)\,\overline{\mathcal{A}_f(n_2)}\\
    &\, \times \sum_{n_3 \in \mathbb{Z}}\, \overline{\mathcal{A}_f(n_3)}\,\mathcal{A}_f(n_2 +n_3-n_1)\,W\left(\frac{n_3-n_1}{ Y}\right)\, V_3\left(\frac{n_3}{ X_2}\right).
\end{align*}
 Since $V_3(n_3/X_2) \equiv 1$ on the effective support of the summation, the insertion of this weight $V_3$ does not change the value of the sum. Now, we detect the equation $n_2+n_3-n_1 = n_4$ with the help of the delta symbol and using the support conditions of $V_1, V_2$ and $V_3$, we introduce the weight function $V_4$ to the $n_4$-sum satisfying the similar properties and support conditions as $V_3$.
With this, we arrive at the following
    \begin{align*}
		\omega(X,Y) \,=& \,\sum_{n_1=1}^\infty\,\mathcal{A}_f(n_1) V_1\left(\frac{n_1}{ X_1}\right)\,\sum_{n_2 =1}^\infty\,\overline{\mathcal{A}_f(n_2)} V_2\left(\frac{n_2}{ X_1}\right)\\
  &\times \, \sum_{n_3 \in \mathbb{Z}}\, \overline{\mathcal{A}_f(n_3)}\,W\left(\frac{n_3-n_1}{ Y}\right)V_3\left(\frac{n_3}{ X_2}\right)
		 \, \\ &\times \sum_{n_4 \in \mathbb{Z}} \,\mathcal{A}_f(n_4) V_4\left(\frac{n_4}{ X_3}\right) \,\delta({n_4 = n_2 + n_3 -n_1 }),
	\end{align*}
    

\noindent where $X_3 = X_2$. Now, using the DFI's expansion of the delta symbol $\delta(n, m)$ given in \autoref{a1} with $Q \asymp \sqrt{ X}$, we can write
\begin{align}\label{L13}
    \omega(X,Y)  =&\, \frac{1}{Q}\,\sum_{q\leqslant Q}\frac{1}{q}\, \  \sideset{}{^\star} \sum_{a\bmod q}\,\int_{\mathbb{R}} \,w(u)\, \psi(q,u)\notag\\
    &\times\, \sum_{n_1 =1}^{\infty}\,\mathcal{A}_f(n_1) \,e\left(\frac{an_1}{q}\right)\,e\left(\frac{u n_1}{qQ}\right)\,V_1\left(\frac{n_1}{ X_1}\right) \notag\\ 
    &\times\, \sum_{n_2 =1}^{\infty}\,\overline{\mathcal{A}_f(n_2)}\, e\left(\frac{-an_2}{q}\right)\,e\left(\frac{-u n_2}{qQ}\right)\,V_2\left(\frac{n_2}{ X_1}\right) \notag\\
    &\times\, \sum_{n_3 =1}^{\infty}\,\overline{\mathcal{A}_f(n_3)}\, e\left(\frac{-an_3}{q}\right)\,e\left(\frac{-u n_3}{qQ}\right)\,V_3\left(\frac{n_3}{ X_2}\right)\,W\left(\frac{n_3-n_1}{ Y}\right) \notag\\
    &\times\, \sum_{n_4 =1}^{\infty}\,\mathcal{A}_f(n_4)\, e\left(\frac{an_4}{q}\right)\,e\left(\frac{u n_4}{qQ}\right)\,V_4\left(\frac{n_4}{ X_3}\right)du + \mathcal{O}(X^{-2025}).
\end{align}
Here, $w$ is a bump function supported on $[-2 X^{\epsilon}, 2 X^{\epsilon}]$ such that $w(u) = 1$, for all $u \in [- X^{\epsilon},  X^{\epsilon}]$, with $w^{(j)}(u) \ll_j 1$. Now we apply the $\mathrm{GL}(2)$ Voronoi summation formula to the $n_1, n_2, n_3$, and $n_4$ sums. We consider the case when $f$ is a holomorphic cusp form. Similar calculations exist in the case of Maass cusp forms with a small difference in the treatment of the integral transforms (see Lemma \ref{voronoi}). We start with the $n_3$-sum, and subsequent calculations will follow the same steps for the $n_1$, $n_2$, and $n_4$ sums. The following lemma provides the necessary details.

\begin{lemma}\label{L14}
		Let $\mathscr{C}_3$ denote the $n_3$-sum as given in \autoref{L13}. Then we have
		\begin{align*}
			\mathscr{C}_3 = \frac{{ X_2^{3/4}}}{\sqrt q}\, \sum_{\pm}\,\sum_{m_3 \ll X^{\epsilon}(q X_2)^2/ Y^2 X_2 } \frac{\overline{\mathcal{A}_f(m_3)}}{m^{1/4}_3}\, e\left( \frac{\overline{a}m_3}{q}\right)\, \mathscr{I}^{\pm}_3 + \mathcal{O}_{A_3}( X^{-A_3}),
		\end{align*} 
  where $A_3$ is a large positive real number and
		\begin{align}\label{zzz}
			&\mathscr{I}^{\pm}_3 =  \int_0^\infty  U_3(x_3)\,W\left(\frac{ X_2x_3-n_1}{ Y}\right)\,  e\left(-\frac{ X_2x_3u}{qQ}\pm  \frac{2\sqrt{ X_2m_3 x_3}}{q} \right)dx_3.
		\end{align}
Here, $U_3$ denotes a new weight function depending upon $V_3$.
\end{lemma}
\begin{proof}
On applying the $\mathrm{GL}(2)$ Voronoi summation formula from Lemma \ref{voronoi} to the $n_3$-sum, we get
		\begin{align}\label{A5}
			\mathscr{C}_3 = \,\frac{X_2}{q} \sum_{\pm} \sum_{m_3 =1}^\infty \,\overline{\mathcal{A}_f(m_3)}\, e\left(\pm \frac{\overline{a}m_3}{q}\right) \,\mathcal{V}^{\pm}_3 \left( \frac{m_3X_2}{q^2}\right),
		\end{align}  
where $\mathcal{V}^{-}_3 = 0$ and
  \begin{align*}
			\mathcal{V}^{+}_3 \left( \frac{m_3X_2}{q^2}\right)=\, 2\pi i^{k}\, \int_0^\infty V_3\left(x_3\right) e\left(-\frac{X_2 u x_3 }{qQ}\right)\,W\left(\frac{X_2n_3-n_1}{ Y}\right)\,J_{k-1 } \left(\frac{4\pi \sqrt{X_2m_3 x_3}}{q}\right) dx_3.
		\end{align*}
  Extracting the oscillations of the Bessel function, i.e., writing
$$ J_{k-1} (4 \pi x) = \frac{1}{\sqrt{4\pi x}}\left(e(2x)\,U_{k-1} (4\pi x)  + e(-2x)\,\overline{U_{k-1}} (4\pi x)\right),$$ where $U_{k-1}$ is a smooth function with  $$x^{j} U_{k-1}^{(j)}(x) \ll_{j} 1, \ \ j \geqslant 0, \ x\gg 1.$$ 
 We obtain
 \begin{align*}
     \mathcal{V}^{+}_3 \left( \frac{m_3X_2}{q^2}\right) = \,\frac{\sqrt{q}}{(X_2 m_3)^{1/4}}\,\sum_{\pm}\,\mathscr{I}^{\pm}_3,  
 \end{align*}
 where $\mathscr{I}^{\pm}_3$ is defined as in \autoref{zzz}. The dual length $m_3 \ll X^{\epsilon}(q X_2)^2/ Y^2 X_2$ is obtained by applying  integration by parts to $\mathscr{I}^{\pm}_3$. Plugging the value of $\mathcal{V}^{+}_3 \left( \frac{m_3X_2}{q^2}\right) $ in \autoref{A5}, we get the lemma.

\end{proof}

Next, we will apply the $\mathrm{GL}(2)$ Voronoi summation formula on the sum over $n_1$. We have the following lemma.

\begin{lemma}\label{L15}
		Let $\mathscr{C}_1$ denote the $n_1$-sum as given in \autoref{L13}. Then we have
		\begin{align*}
			\mathscr{C}_1 = \frac{{ X_1^{3/4}}}{\sqrt q}\, \sum_{\pm}\,\sum_{m_1 \ll X^{\epsilon} (q X_1)^2/ Y^2 X_1 } \frac{\mathcal{A}_f(m_1)}{m^{1/4}_1}\, e\left(-\frac{\overline{a}m_1}{q}\right)\, \mathscr{I}^{\pm}_1 + \mathcal{O}_{A_1}( X^{-A_1}),
		\end{align*} where $A_1$ is a large positive real number and
		\begin{align*}
			&\mathscr{I}^{\pm}_1 =  \int_0^\infty  U_1(x_1)\,W\left(\frac{ X_2x_3 -  X_1x_1}{ Y}\right)\,  e\left(\frac{ X_1x_1u}{qQ}\pm 2 \frac{\sqrt{ X_1x_1m_1}}{q} \right)dx_1.
		\end{align*}
  Here, $U_1$ is a new weight function depending upon $V_1$.
	\end{lemma}
 \begin{proof}
The proof follows from the same steps as Lemma \ref{L14}.
 \end{proof}

\begin{lemma}\label{L16}
		Let $\mathscr{C}_2$ denote the $n_2$-sum as given in \autoref{L13}. Then we have
		\begin{align*}
			\mathscr{C}_2 = \frac{{ X_1^{3/4}}}{\sqrt q}\, \sum_{\pm}\,\sum_{m_2 \ll X^{\epsilon}Q^2/ X_1 } \frac{\overline{\mathcal{A}_f(m_2)}}{m^{1/4}_2}\, e\left(- \frac{\overline{a}m_2}{q}\right)\, \mathscr{I}^{\pm}_2 + \mathcal{O}_{A_2}( X^{-A_2}),
		\end{align*} where $A_1$ is a large positive real number and
		\begin{align*}
			&\mathscr{I}^{\pm}_2 =  \int_0^\infty  U_2(x_2)\,  e\left(-\frac{ X_1x_2u}{qQ}\pm \frac{2\sqrt{ X_1x_2m_2}}{q} \right)dx_2.
		\end{align*}
  Here, $U_2$ is a new weight function depending upon $V_2$.
	\end{lemma}
  \begin{proof}
The proof follows from the same steps as Lemma \ref{L14}.
 \end{proof}

 \begin{lemma}\label{L17}
		Let $\mathscr{C}_4$ denote the $n_4$-sum as given in \autoref{L13}. Then we have
		\begin{align*}
			\mathscr{C}_4 = \frac{{ X_3^{3/4}}}{\sqrt q}\, \sum_{\pm}\,\sum_{m_4 \ll X^{\epsilon}Q^2/ X_3 } \frac{\mathcal{A}_f(m_4)}{m^{1/4}_4}\, e\left(- \frac{\overline{a}m_4}{q}\right)\, \mathscr{I}^{\pm}_4 + \mathcal{O}_{A_4}( X^{-A_4}),
		\end{align*} where $A_4$ is large positive real number and
		\begin{align*}
			&\mathscr{I}^{\pm}_4 =  \int_0^\infty  U_4(x_4)\,  e\left(\frac{ X_3x_4u}{qQ}\pm \frac{2\sqrt{ X_3x_4m_4}}{q} \right)dx_4.
		\end{align*}
   Here, $U_4$ is a new weight function depending upon $V_4$.
	\end{lemma}
  \begin{proof}
The proof follows from the same steps as Lemma \ref{L14}.
 \end{proof}
 \vspace{0.2cm}

Substituting expressions from Lemma \ref{L14} - Lemma \ref{L17} into \autoref{L13}, we get the following expression for $\omega(X,Y)$. 

\begin{align}\label{omega(X,Y)}
\omega(X,Y) = &\frac{ X_1^{3/2}( X_2 X_3)^{3/4}}{Q}\sum_{q\leqslant Q}\frac{1}{q^{3}}\, \,\,  \sum_{\pm}\,\sum_{m_1 \ll X^{\epsilon}(q X_1)^2/ Y^2 X_1 } \frac{\mathcal{A}_f(m_1)}{m^{1/4}_1}\, \sum_{\pm}\,\sum_{m_2 \ll X^{\epsilon}Q^2/ X_1 } \frac{\overline{\mathcal{A}_f(m_2)}}{m^{1/4}_2} \notag\\
 &\times \sum_{\pm}\,\sum_{m_3 \ll X^{\epsilon}(q X_2)^2/ Y^2 X_2 } \frac{\overline{\mathcal{A}_f(m_3)}}{m^{1/4}_3}\,  \sum_{\pm}\,\sum_{m_4 \ll X^{\epsilon}Q^2/ X_3 } \frac{\mathcal{A}_f(m_4)}{m^{1/4}_4}\, \mathcal{C}(...)\mathscr{I}(...) + \mathcal{O}_A( X^{-A}),
\end{align}
where the character sum $\mathcal{C}(...)$ is given by
\begin{align}\label{char sum}
\mathcal{C}(...) = \sideset{}{^\star} \sum_{a \bmod q}\, e\left(- \frac{\overline{a}m_1}{q}\right)\,e\left( \frac{\overline{a}m_2}{q}\right)\,e\left( \frac{\overline{a}m_3}{q}\right)\,e\left(- \frac{\overline{a}m_4}{q}\right),  
\end{align}
which is Ramanujan sum. Thus, it becomes 
\begin{equation} \label{char sum1}
  \mathcal{C}(...) = \sum_{\substack{d \mid q \\ m_1+m_4 \equiv m_2+m_3 \ \textrm{mod } d}} d \mu\left(\frac{q}{d}\right).
\end{equation}
The integral transform $\mathscr{I}(...)$ is given by
\begin{align}\label{integral transform}
\mathscr{I}(...) = &\int_{\mathbb{R}}\,w(u)\, \psi(q,u)\,\,\int_{0}^\infty U_1(x_1)\, e\left(\frac{ X_1x_1u}{qQ} \pm  \frac{2 \sqrt{ X_1m_1x_1}}{q} \right)\notag\\
&\times \,\int_{0}^\infty U_2(x_2)\, e\left(-\frac{ X_1x_2u}{qQ} \pm  \frac{2 \sqrt{ X_1m_2x_2}}{q} \right)\,\notag\\
&\times \,\int_{0}^\infty U_3(x_3)\,W\left(\frac{ X_2x_3- X_1x_1}{ Y}\right) \,\, e\left(- \frac{ X_2x_3u}{qQ} \pm  \frac{2 \sqrt{ X_2m_3x_3}}{q} \right)\notag\\
&\times\,\int_{0}^\infty U_4(x_4)\, e\left( \frac{ X_3x_4u}{qQ} \pm  \frac{2 \sqrt{ X_3m_4x_4}}{q} \right)dx_1\,dx_2\,dx_3\,dx_4\,du.
\end{align}


\subsection{Analysis of the Integral transform:}\label{section int transform} 
Consider the $u$-integral:
\begin{equation*}
    \int_{\mathbb{R}} w(u) \psi(q,u) e\left(\frac{( X_1x_1- X_1x_2- X_2x_3+ X_3x_4)u}{qQ}\right) du,
\end{equation*}
From the properties of $\psi(q,u)$ established in Section \ref{delta}, we have
\begin{equation*}
    \frac{\partial^j}{\partial u^j}\psi(q,u) \ll_j Q^{\varepsilon j} \ \ \text{and} \ \ w^{j}(u) \ll_j Q^{\varepsilon j}, \hspace{1cm} \text{for any} \ j\geqslant1.
\end{equation*}
Therefore, applying repeated integration by parts, the $u$ integral is negligibly small unless
\begin{equation} \label{condtn}
    |x_1-x_2-x_3+x_4| \ll \frac{qQ}{ X}Q^{\varepsilon}. 
\end{equation}
In particular, when $q \gg Q^{1-\varepsilon}$, the condition \eqref{condtn} holds trivially.

 
%

\smallskip

\noindent Let $x_1-x_2-x_3+x_4 =: t$ with $|t| \ll qQ^{1+\varepsilon}/ X$, we arrive at the following expression for $\mathscr{I}(...)$.
\begin{align} \label{Integral transform I}
  \mathscr{I}(...) &=  \int_{|t|\ll \frac{qQ^{1+\varepsilon}}{ X} } \int_0^\infty U_1(x_1) e\left( \pm  \frac{2 \sqrt{ X_1m_1x_1}}{q} \right) 
\int_{0}^\infty U_2(x_2)\, e\left( \pm  \frac{2 \sqrt{ X_1m_2x_2}}{q} \right)  \notag\\
&\times \,\int_{0}^\infty U_3(x_3) W\left(\frac{ X_2x_3- X_1x_1}{ Y}\right) U_4(t-x_1+x_2+x_3) e\left( \pm  \frac{2 \sqrt{ X_2m_3x_3}}{q} \right) \notag\\
& \hspace{3.4cm} \times  e\left( \pm  \frac{2 \sqrt{ X_3m_4(t-x_1+x_2+x_3)}}{q} \right)\,dx_1 dx_2 dx_3 dt.
\end{align}
Next, consider $x_1$ and $x_3$ integrals:
\begin{multline*}
    \int_0^\infty \int_0^\infty U_1(x_1) U_3(x_3) W\left(\frac{ X_2x_3- X_1x_1}{ Y}\right) U_4(t-x_1+x_2+x_3) \\ 
    \times e\left( \pm  \frac{2 \sqrt{ X_1m_1x_1}}{q} \right) e\left( \pm  \frac{2 \sqrt{ X_2m_3x_3}}{q} \right) e\left( \pm  \frac{2 \sqrt{ X_3m_4(t-x_1+x_2+x_3)}}{q} \right) dx_1 dx_3.
\end{multline*}
Applying the change of variable $u= \frac{ X_2x_3- X_1x_1}{ Y}$ and $v=x_3$, the above integrals becomes
\begin{multline*}
    \frac{- Y}{ X_1} \int_0^\infty \int_0^\infty U_1\left(\frac{ X_2v- Yu}{ X_1}\right) U_3(v)W(u) \\ 
    \times U_4\left( t+x_2 + \frac{( X_1- X_2)}{ X_1}v + \frac{ Y}{ X_1}u\right) e\left( \pm\frac{2\sqrt{m_1\left( X_2v- Yu \right)}}{q}\right)e\left( \pm  \frac{2 \sqrt{ X_2m_3v}}{q} \right) \\ \times e\left( \pm  \frac{2 \sqrt{ X_3m_4\left( t+x_2 + \frac{( X_1- X_2)}{ X_1}v + \frac{ Y}{ X_1}u\right)}}{q} \right) du dv.
\end{multline*}
Substituting this back into \autoref{Integral transform I} and treating the remaining integrals trivially, we obtain the following bound
\begin{equation} \label{bound I}
    \mathscr{I}(...) \ll \frac{q Y}{Q^{1-\varepsilon}  X}.
\end{equation}
\subsection{Final estimates}
By plugging the expression for the character sum $\mathcal{C}(...)$ from \autoref{char sum} and bound for the integral transform $\mathscr{I}(...)$ from \autoref{bound I} into \autoref{omega(X,Y)}, we obtain
\begin{multline*}
   \omega(X,Y) \ll \frac{ Y X_1^{3/2}( X_2 X_3)^{3/4}}{Q^{2-\varepsilon} X} \sum_{q \leqslant Q} \frac{1}{q^2} \sum_{m_1 \ll X^{\varepsilon} (q X_1)^2/ Y^2 X_1 } \frac{|\mathcal{A}_f(m_1)|}{m^{1/4}_1} \sum_{m_2 \ll X^{\varepsilon} Q^2/ X_1 } \frac{\overline{|\mathcal{A}_f(m_2)|}}{m^{1/4}_2} \\
 \times \sum_{m_3 \ll X^{\varepsilon} (q X_2)^2/ Y^2 X_2 } \frac{\overline{|\mathcal{A}_f(m_3)|}}{m^{1/4}_3} \sum_{m_4 \ll X^{\varepsilon} Q^2/ X_3 } \frac{|\mathcal{A}_f(m_4)|}{m^{1/4}_4} \sum_{\substack{d \mid q \\ m_1+m_4 \equiv m_2 + m_3 \ \textrm{mod } d}} d \mu\left(\frac{q}{d}\right).
\end{multline*}
For a fixed $m_1,m_2$, $m_3$, the variable $m_4$ can be determined. We can bound the above expression by:
\begin{multline*}
    \ll \frac{ Y X_1^{3/2}( X_2 X_3)^{3/4}}{Q^{2-\varepsilon} X} \sum_{q \leqslant Q} \frac{1}{q^2} \sum_{m_1 \ll X^{\varepsilon} (q X_1)^2/ Y^2 X_1 } \frac{|\mathcal{A}_f(m_1)|}{m^{1/4}_1} \sum_{m_2 \ll X^{\varepsilon} Q^2/ X_1 } \frac{\overline{|\mathcal{A}_f(m_2)|}}{m^{1/4}_2} \\
 \times \sum_{m_3 \ll X^{\varepsilon} (q X_2)^2/ Y^2 X_2 } \frac{\overline{|\mathcal{A}_f(m_3)|}}{m^{1/4}_3} \sum_{m_4 \ll X^{\varepsilon} Q^2/ X_3 } \frac{|\mathcal{A}_f(m_4)|}{m^{1/4}_4} (q, m_1+m_4 - m_2 -m_3),
\end{multline*}
split the sum over $q$ into dyadic intervals of size $\mathcal{Q}$ such that $\mathcal{Q} \ll Q$. The above expression becomes bounded by
\begin{multline*}
    \ll \frac{ Y X_1^{3/2}( X_2 X_3)^{3/4}}{Q^{2-\varepsilon} X} \sup_{\mathcal{Q} \ll Q} \sum_{m_1 \ll X^{\varepsilon} (\mathcal{Q} X_1)^2/ Y^2 X_1 } \frac{|\mathcal{A}_f(m_1)|}{m^{1/4}_1} \sum_{m_2 \ll X^{\varepsilon} Q^2/ X_1 } \frac{\overline{|\mathcal{A}_f(m_2)|}}{m^{1/4}_2} \\
    \times \sum_{m_3 \ll X^{\varepsilon} (\mathcal{Q} X_2)^2/ Y^2 X_2 } \frac{\overline{|\mathcal{A}_f(m_3)|}}{m^{1/4}_3} \sum_{m_4 \ll X^{\varepsilon} Q^2/ X_3 } \frac{|\mathcal{A}_f(m_4)|}{m^{1/4}_4} \sum_{q \sim \mathcal{Q}} \frac{1}{q^2} (q, m_1+m_4 - m_2 -m_3).
\end{multline*}
Using the gcd on average bound for the $q$ sum, we get
\begin{multline*}
    \ll \frac{ X^\varepsilon Y X_1^{3/2}( X_2 X_3)^{3/4}}{Q^{2-\varepsilon} X} \sup_{\mathcal{Q} \ll Q} \frac{1}{\mathcal{Q}} \sum_{m_1 \ll X^{\varepsilon} (\mathcal{Q} X_1)^2/ Y^2 X_1 } \frac{|\mathcal{A}_f(m_1)|}{m^{1/4}_1} \sum_{m_2 \ll X^{\varepsilon} Q^2/ X_1 } \frac{\overline{|\mathcal{A}_f(m_2)|}}{m^{1/4}_2} \\
    \times \sum_{m_3 \ll X^{\varepsilon} (\mathcal{Q} X_2)^2/ Y^2 X_2 } \frac{\overline{|\mathcal{A}_f(m_3)|}}{m^{1/4}_3} \sum_{m_4 \ll X^{\varepsilon} Q^2/ X_3 } \frac{|\mathcal{A}_f(m_4)|}{m^{1/4}_4}.
\end{multline*}
Further, applying Cauchy's inequality to each of these sums along with the bound for $\mathrm{GL}(2)$ Fourier coefficients given in \autoref{bound for gl2 coeff}, we arrive at
\begin{align*}
  \omega(X,Y)   &\ll \frac{ Y X_1^{3/2}( X_2 X_3)^{3/4}}{Q^{2-\varepsilon} X} \sup_{\mathcal{Q} \ll Q} \frac{1}{\mathcal{Q}} \left( \frac{(\mathcal{Q} X_1)^2}{ Y^2 X_1}\right)^{3/4} \left(\frac{Q^2}{ X_1}\right)^{3/4} \left( \frac{(\mathcal{Q} X_2)^2}{ Y^2 X_2}\right)^{3/4}  \left(\frac{Q^2}{ X_3}\right)^{3/4} \\
    &\ll \frac{ X^{\varepsilon} (X_1 X_2)^{3/2} Q^{1+\varepsilon}}{ X} \frac{1}{Y^2} \sup_{\mathcal{Q} \ll Q} \frac{1}{\mathcal{Q}} \mathcal{Q}^3 \ll \frac{X^2 Q^{3+\varepsilon}}{Y^2} = \frac{ X^{7/2 +\varepsilon}}{Y^2},
\end{align*}
as $Q = \sqrt{ X}$ and $X_i \asymp X$ for each $i=1,2,3$.
Substituting this bound for $\omega(X,Y)$ into \autoref{kappa to omega} and further putting it into \autoref{S( X, Y) with kappa}, we have the following:
\begin{align}
    \mathscr{S}_{ X, Y} \ll \norm{\mathsf{a}}_2 \norm{\mathsf{b}}_2 Y^{1/4} \left(  \frac{ X^{7/2 +\varepsilon}}{Y^2}\right)^{1/4} = \norm{\mathsf{a}}_2 \norm{\mathsf{b}}_2 \frac{ X^{7/8 + \varepsilon}}{ Y^{1/4}}.
\end{align}
The above bound is non-trivial provided 
\begin{equation}
     \frac{ X^{7/8 + \varepsilon}}{ Y^{1/4}} < ( XY )^{1/2+\varepsilon} \iff  X^{1/2} < Y <  X^{1-\varepsilon}.
\end{equation}
\vspace{0.3cm}
\section{Proof of Theorem \ref{th2}}

In Theorem \ref{th2}, we aim to prove the cancellations in the following sum
\begin{align*}
\mathcal{T}_{X, \mathsf{S}} =  \sum_{X \leqslant n \leqslant 2X}\,\sum_{m \in \mathsf{S}} \, \mathsf{c}(n)\, \mathcal{A}_{f}(n+m),
\end{align*}
On applying the Cauchy-Schwarz inequality to the sum over $n$, we get
 \begin{align} \label{TXS}
\mathcal{T}_{X, \mathsf{S}} &\leqslant \norm{\mathsf{c}}_2 \lr{\sum_{n \in \mathbb{Z}} W_1\lr{\frac{n}{X}}\left|\sum_{m\in S} \mathcal{A}_f(n+m)\right|^2}^{1/2} = \norm{\mathsf{c}}_2  \,  \mathcal{U}^{1/2}_{X, \mathsf{S}},  
 \end{align}
where $W_1$ is positive suitable bump function, supported on the interval $[-2,2]$ with $W_1(x) = 1$ for $x \in [-1,1]$, also satisfies $x^j W_1^{(j)}(x) \,\ll_j\,1$, for $j\geqslant 0$ and
 \begin{align*}
 \mathcal{U}_{X, \mathsf{S}}  =\,\sum_{n \in \mathbb{Z}}\,  \sum_{m_1 \in \mathsf{S}}\,\sum_{m_2 \in \mathsf{S}}\,\, \mathcal{A}_{f}( n+m_1)\,\overline{\mathcal{A}_{f}(n+m_2)} \,W_1\left(\frac{n}{X}\right).
\end{align*}
Changing the variables $n + m_2 \longrightarrow n$, we can write
\begin{align*}
 \mathcal{U}_{X, \mathsf{S}}  =  \sum_{n \in \mathbb{Z}}\,\overline{\mathcal{A}_{f}( n)}\,\,\sum_{m_1 \in \mathsf{S}}\,\sum_{m_2 \in \mathsf{S}}\,\, {\mathcal{A}_{f}(n+m_1-m_2)}\,W_1\left(\frac{n-m_2}{X}\right).    
\end{align*} 
Applying the Cauchy-Schwarz inequality to the sum over $n$, we get that $  \mathcal{U}_{X, \mathsf{S}}$ is dominated by
\begin{align*}
 \left(\sum_{n \ll X}\,\left|\overline{\mathcal{A}_{f}( n)}\right|^2\right)^{1/2}\,   \left(\sum_{n \ll X}\,\left|\,\sum_{m_1 \in \mathsf{S}}\,\sum_{m_2 \in \mathsf{S}}\,{\mathcal{A}_{f}( n+m_1-m_2)}\,W_1\left(\frac{n-m_2}{X}\right)\right|^2\right)^{1/2}.
\end{align*}
By using the Ramanujan bound on average from \autoref{bound for gl2 coeff}, we get
\begin{align*}
 \mathcal{U}_{X, \mathsf{S}} \ll&\, X^{1/2}\,\left(\sum_{n \ll X}\,\sum_{m_1 \in \mathsf{S}}\,\sum_{m_2 \in \mathsf{S}}\, {\mathcal{A}_{f}(n+m_1-m_2)}\,W_1\left(\frac{n-m_2}{X}\right) \right.\\
& \hspace{3.5cm}\left.\, \times \sum_{m_3 \in \mathsf{S}}\,\sum_{m_4 \in \mathsf{S}}\, \overline{\mathcal{A}_{f}(n+m_3-m_4)}\,W_2\left(\frac{n-m_4}{X}\right)\right)^{1/2},      
\end{align*}
Since $W_1$ is a real-valued function, $W_2 = W_1$. Here, we use different notations for the copy of $W_1$ to be more precise in our steps. Changing the variable $n+m_3-m_4 \longrightarrow n$, we get
\begin{multline*}
 \mathcal{U}_{X, \mathsf{S}} \ll\, X^{1/2}\,  \Biggl( \ \sum_{n \ll X}\,\overline{\mathcal{A}_{f}(n)}\,\,\mathop{\sum_{m_i \in \mathsf{S}}}_{1 \leqslant i \leqslant 2^2}\,{\mathcal{A}_{f}(n+m_1-m_2-m_3+m_4)}\\
\times \,W_1\left(\frac{n-m_2-m_3+m_4}{X}\right)\,W_2\left(\frac{n-m_3}{X}\right) \Biggr)^{1/2}.   
\end{multline*}
We again apply the Cauchy-Schwarz inequality to the $n$-sum then use the Ramanujan bound on average for the $\mathrm{GL}(2)$ Fourier coefficients $\mathcal{A}_{f}(n)$, and again after a change of variables, we get that $ \mathcal{U}_{X, \mathsf{S}}$ is dominated by the following expression 
\begin{align*}
X^{1/2 + 1/2^2}\, &\Biggl(\ \sum_{n \ll X}\,\overline{\mathcal{A}_{f}( n)}\,\,\mathop{\sum_{m_i \in \mathsf{S}}}_{1 \leqslant i \leqslant 2^3}\,{\mathcal{A}_{f}(n+m_1-m_2-m_3+m_4-m_5+m_6+m_7-m_8)}  \\
&\hspace{3.5cm}  \times\,W_1\left(\frac{n-m_2-m_3+m_4-m_5 +m_6+m_7-m_8}{X}\right) \, \\
 &\hspace{3.5cm} \,\times\,W_2\left(\frac{n-m_3-m_5+m_6+m_7-m_8}{X}\right)\\
 &\hspace{3.5cm}\,\times \,W_3\left(\frac{n-m_5}{X}\right)\,W_4\left(\frac{n-m_5+m_6-m_8}{X}\right) \Biggr)^{1/2^2}.    
\end{align*}

\noindent In the argument of $\mathcal{A}_{f}(n)$ in the above expression, we can see that half of $m_i$'s appear with negative signs and the other half with positive signs. For simplicity, we make the change of variables and note that $ \mathcal{U}_{X, \mathsf{S}}$ is dominated by the following expression
\begin{multline*}
X^{1/2 + 1/2^2} \left(\sum_{n \ll X}\,\overline{\mathcal{A}_{f}( n)}\,\,\mathop{\sum_{m_i \in \mathsf{S}}}_{1 \leqslant i \leqslant 2^3} \,\mathcal{A}_{f}(n+ \sum_{j=0}^3 (m_{2j+1} - m_{2j+2})) \right.\\ \hspace{3cm} \times \left. \prod_{1 \leqslant j \leqslant 2^2}\,W_j\left(\frac{n + \mathop{\sum}_{1 \leqslant s \leqslant 2^3}\,\epsilon_{js} m_s}{X}\right)\,\right)^{1/2^2},    
\end{multline*}
where $\epsilon_{js} \in \left\{0, 1, -1\right\}$. Let $\mathsf{A} = \mathsf{S}-\mathsf{S}: = \{s_1 - s_2 : s_1 \in \mathsf{S}, s_2 \in \mathsf{S}\} \subset [-X, X]$. We repeat the same treatment as above to the $n$-sum, $\ell$-many times such that $\ell$ is least with $ |2^{\ell} \mathsf{A}| \gg X^{1-\delta_1}$ for some $\delta_1 \geqslant 1/2$ and $|2^{\ell} \mathsf{A}| \ll X^{2^{\ell+1}/k} < X^{1-\epsilon}$ for any epsilon $0<\epsilon  < 1/2$. We finally get that $ \mathcal{U}_{X, \mathsf{S}}$ is dominated by
\begin{align*}
 &\,X^{1/2 + 1/2^2 + ... + 1/2^{\ell}}\\
 &\times\left(\sum_{{n \ll X}}\overline{\mathcal{A}_{f}( n)}\hspace{-0.2cm}\mathop{\sum_{m_i \in \mathsf{S}}}_{1 \leqslant i \leqslant 2^{\ell+1}}{\mathcal{A}_{f}(n+ \sum_{j=0}^{2^\ell-1} (m_{2j+1} - m_{2j+2})} ) \prod_{1 \leqslant j \leqslant 2^{\ell}}W_j\left(\frac{n + \mathop{\sum}_{1 \leqslant s \leqslant 2^{\ell +1}}\,\epsilon_{js} m_s}{X}\right)   \right)^{1/2^{\ell}}. 
 \end{align*}
 Here, $W_j = W_1$ for all $1 \leqslant j \leqslant 2^{\ell}$. Using the inverse Fourier transform for the function $W_j$ for all $1 \leqslant j \leqslant 2^{\ell}$, we get that $ \mathcal{U}_{X, \mathsf{S}}$ is dominated by
 \begin{align*}
X^{1 - 1/2^{\ell}}\,  &\Biggl( \ \ \int_{|y_1| \ll X^{\varepsilon }} \int_{|y_2| \ll X^{\varepsilon } }...\int_{|y_{2^{\ell}}| \ll X^{\varepsilon }}\,\widehat{W_1}(y_1)\,\widehat{W_2}(y_1)...\widehat{W_{2^{\ell}}}(y_{2^{\ell} } )\,\,  \\
 &  \times\,\,\,
  \sum_{n \ll X}\,\overline{\mathcal{A}_{f}(n) }\,e\left(\frac{n (y_1 + y_2 + ... + y_{2^{\ell}} ) }{X}\right)\,\mathop{\sum_{m_i \in \mathsf{S}} }_{1 \leqslant i \leqslant 2^{\ell+1} } \,\mathcal{A}_{f}(n+\sum_{j=0}^{2^\ell-1} (m_{2j+1} - m_{2j+2})) \\ 
 &\,\times\,\,\,\prod_{1 \leqslant j \leqslant 2^{\ell}}\,e\left(\frac{\left(\mathop{\sum}_{1 \leqslant s \leqslant 2^{\ell +1} }\,\epsilon_{sj} m_s \right)y_j}{X} \right) \Biggr)^{1/2^{\ell}} ,   
\end{align*}
where $\widehat{W_j}$ denotes the Fourier transform of the function $W_j$ for all $1 \leqslant j \leqslant 2^{\ell}$. Define
 \begin{align*}
 \alpha(n): = \overline{\mathcal{A}_{f}(n)}\,e\left(\frac{n(y_1 + y_2 + ... + y_{2^{\ell}})}{X}\right),    
 \end{align*} 
 and
 \begin{align*}
 \beta_{\ell}(m) :=  \mathop{\sum_{m_1 \in \mathsf{S}}\,\sum_{m_2 \in \mathsf{S}}\,...\,\sum_{m_{2^{\ell+1}} \in \mathsf{S}}}_{m \,=\,(m_1-m_2) + (m_3 - m_4) \,+\, ...\, + \,(m_{2^{\ell+1}-1}- m_{2^{\ell+1}}) } \, \prod_{1 \leqslant j \leqslant 2^{\ell}}\,e\left(\frac{\left(\mathop{\sum}_{1 \leqslant s \leqslant 2^{\ell +1}}\,\epsilon_{sj} m_s\right)y_j}{X}\right). 
 \end{align*}
 Now, we can write
\begin{align} \label{T1XS}
 \mathcal{U}_{X, \mathsf{S}} \ll X^{1- 1/2^{\ell} +\varepsilon}\left( \ \int_{|y_1| \ll X^{\varepsilon }}\int_{|y_2| \ll X^{\varepsilon }}...\int_{|y_{2^{\ell}}| \ll X^{\varepsilon }}\widehat{W_1}(y_1)\,\widehat{W_2}(y_1)...\widehat{W_{2^{\ell}}}(y_{2^{\ell}}) \mathcal{V}_{X, \mathsf{S}}\right)^{1/2^{\ell}}, 
\end{align}
where $ \mathcal{V}_{X, \mathsf{S}} $ represents a bilinear sum given by
\begin{align} \label{T2XS}
 \mathcal{V}_{X, \mathsf{S}} = \,  \sum_{n \ll X}\,\sum_{m \leqslant |2^{\ell} \mathsf{A}|}\, \alpha(n)\,\beta_{\ell}(m)\,\mathcal{A}_{f}(n+m) . 
\end{align}
We note that the trivial bound for $ \mathcal{V}_{X, \mathsf{S}}$ gives rise to 
\begin{align*}
     \mathcal{V}_{X, \mathsf{S}} \ll   X^{1+\varepsilon} |2^{\ell} \mathsf{A}| \ll  X^{1+\varepsilon} \left|   \mathsf{A}\right|^{2^{\ell}} \ll  X^{1+\varepsilon} \left|   \mathsf{S}\right|^{2^{\ell+1}}. 
\end{align*}
Note that the Fourier transform $\widehat{W_j}$ is bounded as $W_j$ is a compactly supported smooth function for every $1 \leqslant j \leqslant 2^{\ell}$. Substituting the trivial bound for $   \mathcal{V}_{X, \mathsf{S}}$ in \autoref{T1XS} we obtain that
\begin{align*}
     \mathcal{U}_{X, \mathsf{S}} \ll \,X^{1- 1/2^{\ell} + \varepsilon }\, \left( X^{1+\varepsilon } \left|   \mathsf{S}\right|^{2^{\ell+1}}\right)^{1/2^{\ell}} \ll X^{1+\varepsilon}|\mathsf{S}|^2. 
\end{align*}
Substituting the bound for $\mathcal{U}_{X, \mathsf{S}}$ in \autoref{TXS}, we obtain that
\begin{align*}
     \mathcal{T}_{X, \mathsf{S}} \ll \norm{\mathsf{c}}_2 X^{\frac{1}{2} + \varepsilon} |\mathsf{S}|. 
\end{align*}
We observe that a trivial bound for $  \mathcal{V}_{X, \mathsf{S}}$ gives rise to a trivial bound for $ \mathcal{T}_{X, \mathsf{S}}$. We shall apply Theorem  \ref{th1} to get a non-trivial bound for  $ \mathcal{T}_{X, \mathsf{S}}$.
\vspace{0.3cm}

\noindent Since we have chosen $\ell$ (least) such that $|2^{\ell} \mathsf{A}| \gg X^{1-\delta_1}$, i.e.,
$X^{1-\delta_1} \ll |2^{\ell} \mathsf{A}| \ll X^{2^{\ell+1}/k}$, which implies that $ 2^{\ell+1} > k(1-\delta_1)$. So, we can see that the above sum $ \mathcal{V}_{X, \mathsf{S}}$  is the same as our main sum $\mathscr{S}_{X, Y}$ in Theorem \ref{th1} with $\mathsf{a}(n)=\alpha(n)$ and $\mathsf{b}(m) = \beta_\ell(m)$. By using Theorem \ref{th1} with $ \norm{\alpha}_2 \ll {X}^{1/2}$, we finally get 
\begin{align*}
  \mathcal{T}_{X, \mathsf{S}} & \ll\, X^{\varepsilon } \norm{\mathsf{c}}_2\, X^{1/2- 1/2^{\ell+1}} \,\left(\norm{\alpha}_2\, \norm{\beta_{\ell}}_2\, \frac{X^{3/8}}{\left(X^{2^{\ell+1}/k}\right)^{3/4}}\, \left(X\,X^{2^{\ell+1}/k}\right)^{1/2+\varepsilon }\right)^{1/2^{\ell +1}} \\
 & \ll\, X^{\varepsilon }  \norm{\mathsf{c}}_2\, X^{\frac{1}{2}}\,  X^{\frac{3}{ 2^{\ell +4}}  }\,\left(\, \frac{\norm{\beta_{\ell}}_2}{\left(X^{2^{\ell+1}/k}\right)^{1/4}} X^{\varepsilon }\right)^{1/2^{\ell +1}}\\
 & \ll\,  \norm{\mathsf{c}}_2 X^{\frac{1}{2}} \left(\norm{\beta_{\ell}}_2\right)^{1/2^{\ell+1}}\,  X^{ \frac{3}{2^{\ell +4}}- \frac{1}{4k} + \varepsilon }.
\end{align*}
This bound for $ \mathcal{T}_{X, \mathsf{S}}$ is non-trivial when
\begin{align*}
   \left(\norm{\beta_{\ell}}_2\right)^{1/2^{\ell+1}}\,  X^{\frac{3}{2^{\ell +4}}- \frac{1}{4k} + \varepsilon } < X^{\frac{1}{k}+ \varepsilon }.
\end{align*}
This proves Theorem \ref{th2}.
\vspace{0.4cm}

\noindent \textbf{Acknowledgement:}
  The authors express their gratitude to Prof. Ritabrata Munshi and Prof. Saurabh Kumar Singh for their invaluable suggestions and comments. H. Chanana thanks the Department of Mathematics and Statistics at IIT Kanpur for its excellent academic environment. She is also grateful for the support received from UGC, Government of India. M. Harun wants to thank the Stat-Math Unit at ISI Kolkata for providing the Visiting Scientist fellowship and a supportive research atmosphere for drafting this work. He also wants to thank the Department of Science and Mathematics at the Indian Institute of Information Technology, Guwahati, for providing good research facilities to complete the work. 
  \vspace{0.4cm}

  \noindent \textbf{Declarations:}
  \vspace{0.2cm}

\noindent \textbf{Conflict of interest statement.} On behalf of all authors, the corresponding author states
that there is no Conflict of interest.
\vspace{0.3cm}

\noindent \textbf{Data availability statement.} Data sharing does not apply to this article as no datasets
were generated or analyzed during the current study.

\end{document}